\documentclass[conference,twoside,10pt]{ieeeconf}
\usepackage{generic}
\usepackage{comment}
\usepackage{cite}
\usepackage{hyperref}
\usepackage{amsmath,amssymb,amsfonts}
\usepackage{tikz}
\usetikzlibrary{arrows.meta, patterns, decorations.markings}
\usepackage{algorithmic}
\usepackage{graphicx}
\usepackage{subcaption}
\usepackage{xcolor}
\usepackage{comment}
\newtheorem{theorem}{Theorem}
\newtheorem{corollary}{Corollary}
\newtheorem{property}{Property}
\newtheorem{remark}{Remark}
\newtheorem{example}{Example}
\newtheorem{definition}{Definition}
\newtheorem{ass}{Assumption}
\newtheorem{lem}{Lemma}

\newenvironment{assumption}{\begin{ass}}{\hfill $\bullet$ \end{ass}}
\newenvironment{definitionn}{\begin{definition}}{\hfill $\bullet$ \end{definition}}

\newenvironment{Remark}{\begin{remark}}{\hfill $\bullet$ \end{remark}}
\newenvironment{Theorem}{\begin{theorem}}{\hfill $\square$ \end{theorem}}

\newenvironment{Property}{\begin{property}}{\hfill \end{property}}
\newenvironment{lemma}{\begin{lem}}{\hfill $\square$ \end{lem}}

\usepackage{algorithm,algorithmic}
\usepackage{hyperref}

\usepackage{textcomp}
\def\BibTeX{{\rm B\kern-.05em{\sc i\kern-.025em b}\kern-.08em
    T\kern-.1667em\lower.7ex\hbox{E}\kern-.125emX}}
\begin{document}
\title{\bf Boundary Control for Wildfire Mitigation}
\author{M. C. Belhadjoudja, M. Maghenem, E. Witrant, D. Georges
\thanks{The authors are with Univ. Grenoble Alpes, CNRS, Grenoble-INP, GIPSA-lab, F-38000, Grenoble, France (e-mail: mohamed.belhadjoudja,mohamed.maghenem,didier.georges@gipsa-lab.fr). E. Witrant is also with the Departement of Mechanical Engineering, Dalhousie University, Halifax B3H 4R2, Nova Scotia, Canada (e-mail: emmanuel.witrant@gipsa-lab.fr).}}

\maketitle

\pagenumbering{gobble}

\begin{abstract}
In this paper, we propose a feedback control strategy to protect vulnerable areas from wildfires. We consider a system of coupled partial differential equations (PDEs) that models heat propagation and fuel depletion in wildfires and study two cases. First, when the wind velocity is known, we design a Neumann-type boundary controller guaranteeing that the temperature of some protected region converges exponentially, in the $L^2$ norm, to the ambient temperature. Second, when the wind velocity is unknown, we design an adaptive Neumann-type boundary controller guaranteeing the asymptotic convergence, in the $L^2$ norm, of the temperature of the protected region to the ambient temperature. In both cases, the controller acts along the boundary of the protected region and relies solely on temperature measurements along that boundary. Our results are supported by numerical simulations.
\end{abstract}

\section{Introduction}
Wildfires represent a catastrophic threat, causing widespread destruction across ecosystems, infrastructures, and human settlements. The devastation extends beyond the immediate damage, as wildfires release large amounts of pollutants into the atmosphere, significantly impacting air quality, biodiversity, and soil integrity \cite{air_quality}. As climate change accelerates, the frequency and intensity of these fires are expected to increase, amplifying their effects \cite{worse,worse2}. This trend makes the ability to predict the spread of heat, and protect vulnerable areas from damage, an urgent global challenge \cite{danger1}.

In response to this growing threat, the physics community has developed advanced models to simulate wildfire dynamics, capturing key phenomena such as heat transfer, fuel depletion, and wind effects. A comprehensive survey of physical and quasi-physical models for wildfire propagation, covering research from 1990 to 2007, is provided in \cite{model1}; see also \cite{model_learning} for machine learning-based models. Among the most sophisticated are models based on extended irreversible thermodynamics \cite{model2}, which explicitly account for the coupling between pyrolysis, gas-phase combustion, and heat transfer mechanisms. While these models offer a deep understanding of wildfire physics, their complexity comes at a cost: their intricate structure poses significant challenges for real-time estimation and control. 

Among the various models available, we consider in this paper the framework proposed in \cite{model3}, which strikes a balance between physical relevance and complexity. This model consists of two coupled 2-dimensional PDEs, one governing the transport and diffusion of fire temperature due to fuel combustion, and the other describing the depletion of fuel over time. To the best of our knowledge, this model has been studied by the control community only in \cite{georges,georges2}, where it is used to estimate ignition locations using a sparse sensor network for temperature measurements.

In this paper, we focus on a goal that is different from \cite{georges,georges2}. We suppose that the wildfire, described by the model from \cite{model3}, evolves in some domain $\Omega \subset \mathbb{R}_{+}^2$, with $\mathcal{W}\subset \Omega$ a region to protect. We first consider the case where the wind velocity at the boundary of $\mathcal{W}$ is known, and we design a Neumann-type controller at that location, which ensures that the temperature of $\mathcal{W}$ converges (exponentially) towards the desired ambient level. Next, we suppose that the wind velocity is unknown and bounded, with an unknown upperbound, and we design an adaptive Neumann-type controller which guarantees the boundedness and asymptotic convergence to the ambient temperature of the temperature of $\mathcal{W}$. The only output available in both cases is the temperature along the boundary of $\mathcal{W}$.

The rest of the paper is organized as follows. In Section \ref{problem}, we recall the wildfire model in \cite{model3}, and describe our control objective. In Section \ref{main results}, we state and prove our main results. In Section \ref{simu}, we illustrate our result by performing a set of numerical simulations. The paper is finalized by a conclusion and some perspectives.

\textit{Notation.}
Let $\Omega \subset \mathbb{R}_{+}^2$ be connected. Given $T:\Omega \times [0,+\infty)\to \mathbb{R}$, with $(x,t)\mapsto T(x,t)$ and $x:=\begin{bmatrix} x_1 & x_2 \end{bmatrix}$, we denote by $\Delta T$ the Laplacian of $T$, defined as 
$\Delta T := \frac{\partial^2 T}{\partial x_1^2} + \frac{\partial^2 T}{\partial x_2^2},$
where $\frac{\partial^2 T}{\partial x_i^2}$, for $i\in \{1,2\}$, denotes the second-order partial derivative of $T$ with respect to $x_i$. Furthermore, we let $\nabla T$ be the gradient of $T$, given by 
$\nabla T := \begin{bmatrix}
    \frac{\partial T}{\partial x_1} & \frac{\partial T}{\partial x_2}
\end{bmatrix},$
where $\frac{\partial T}{\partial x_i}$, for $i\in \{1,2\}$, is the partial derivative of $T$ with respect to $x_i$. The divergence of $v : \Omega \to \mathbb{R}^2$, with $v  := \begin{bmatrix}
    v_1  & v_2 
\end{bmatrix}$, is defined as 
$ \nabla \cdot v := \frac{\partial v_1}{\partial x_1} + \frac{\partial v_2}{\partial x_2}.$
We denote by $a\cdot b$ the scalar product between the two vectors $a,b\in \mathbb{R}^2$, and we let $|a|:=\sqrt{a\cdot a}$ be the norm of $a$. Next, we let $L^{2}(\Omega;I)$, for a given set $I\subset \mathbb{R}$, be the space of functions $T:\Omega \to I$ such that $|T|_{L^2(\Omega)}^2 := \int_{\Omega}T ^2dx<+\infty$. We also denote by $H^2(\Omega)$ the space of functions $T:\Omega \to \mathbb{R}$ such that $T,\left|\nabla T\right|$, and $\Delta T\in L^2(\Omega;\mathbb{R})$. We let $L^2_{loc}(0,+\infty;H^2(\Omega))$ be the space of functions $T:\Omega \times (0,+\infty)\to \mathbb{R}$ such that $T(\cdot,t)\in H^2(\Omega)$ for almost all $t\in (0,+\infty)$, and $\int_{K}\int_{\Omega}(T(x,t)^2+\left|\nabla T(x,t)\right|^2+(\Delta T(x,t))^2)dxdt<+\infty$ for any compact set $K\subset (0,+\infty)$. We similarly define the space $L^{2}_{loc}(0,+\infty;L^2(\Omega))$. Finally, a.a. means almost all.

\section{Problem Formulation}\label{problem}
\subsection{Wildfire propagation model}
We consider in this paper the following system \cite{model3}
\begin{equation*}
\Sigma \ : \ \left\lbrace
\begin{aligned}
\frac{\partial T}{\partial t} =&~ \varepsilon \Delta T - v\cdot \nabla T  + A\left(Sr(T)-C(T-T_a)\right) \\
\frac{\partial S}{\partial t} =&~ -C_S Sr(T),
\end{aligned}
\right.
\end{equation*}
where $(T,S) : \Omega \times [0,+\infty) \to \mathbb{R}_{+}\times [0,1]$ and $\Omega \subset \mathbb{R}_{+}^2$ is bounded and connected. The meaning of each term in $\Sigma$ is explained below.
\begin{itemize}
    \item The domain $\Omega$ represents the area where the fire evolves.
    \item $t$ is the time variable.
    \item $T$ is the temperature of the fire.
    \item $S$ represents the fuel supply mass fraction. 
    \item The constant $T_a>0$ is the ambient temperature.
    \item The parameter $\varepsilon >0$ represents the thermal diffusivity, and $\varepsilon \Delta T$ is the diffusion term, governing the spread of heat within the domain.
    \item The coefficient $A>0$ corresponds to the maximal temperature increase per unit time when combustion is at its peak, assuming $S:=1$ (ample fuel supply) and negligible cooling effects.
    \item The constant $C>0$ characterizes the rate at which heat dissipates into the surroundings.
    \item $C_S\geq 0$ characterizes the rate at which fuel diminishes relative to its initial amount.
    \item $v:\Omega \times [0,+\infty)\to \mathbb{R}^2$, $(x,t)\mapsto v(x,t) = \begin{bmatrix} v_1(x,t) & v_2(x,t) \end{bmatrix}$ is the wind velocity, which advects the heat via the advection term $-v\cdot \nabla T$.
    \item $Sr(T)$ represents the rate at which the fuel is consumed due to burning, with $r(T)$ given by the Arrhenius law 
\begin{equation*}
r(T) :=
\left\lbrace 
\begin{aligned}
&\exp^{-\gamma/(T-T_a)} \quad &T>T_a \\
&0 \quad &T\leq T_a,
\end{aligned}
\right.
\end{equation*}
where $\gamma >0$. We refer to $A(Sr(T)-C(T-T_a))$ as the reaction term.
\end{itemize}

In this model, the temperature is in Kelvin, the time is in seconds, and the spatial variables are in meter. The units of the coefficients are detailed in \cite{model3}.

We let $\mathcal{W}\subset \Omega$ be connected, and with a piecewise $\mathcal{C}^1$ boundary $\partial \mathcal{W}$. Along $\partial \mathcal{W}$, we impose the following Neumann-type boundary condition 
\begin{align}
\frac{\partial T(x,t)}{\partial n(x) } =&~ \kappa (T(x,t)-T_a,x,t) \nonumber \\
&~\qquad \text{for a.a. $x\in \partial \mathcal{W}$ and a.a. $t>0$}, \label{neumann_control}
\end{align}
where $\kappa : \mathbb{R}\times \partial \mathcal{W}\times [0,+\infty)\to \mathbb{R}$ will be designed, and $n:\partial \mathcal{W}\to \mathbb{R}^2$ is the outward-pointing unit normal vector field to $\mathcal{W}$; see \cite[Appendix C]{evans}\footnote{We recall that $\frac{\partial T}{\partial n} := n\cdot \nabla T$.}.

Furthermore, we suppose that $\Omega$ admits a piecewise $\mathcal{C}^1$ boundary $\partial \Omega$. Along $\partial \Omega\setminus \partial \mathcal{W}$, we impose the following homogeneous Neumann boundary condition 
\begin{align}
\frac{\partial T(x,t)}{\partial \nu(x) } = 0 \quad \text{for a.a. $x\in \partial \Omega \setminus \partial \mathcal{W}$ and a.a. $t>0$}, \label{neumann}
\end{align}
where $\nu : \partial \Omega\to \mathbb{R}^2$ is the outward-pointing unit normal vector field to $\Omega$.

The solutions to $\Sigma$ under \eqref{neumann_control}-\eqref{neumann} are understood in the following sense. 
\begin{definitionn}\label{def1}
A solution to $\Sigma$ under \eqref{neumann_control}-\eqref{neumann}, starting from the initial condition 
$(T_o,S_o) \in L^2(\Omega;\mathbb{R}_+)\times L^2(\Omega;[0,1]),$  
is any pair 
$(T,S)\in L^2_{loc}(0,+\infty;H^2(\Omega))\times L^2_{loc}(0,+\infty;L^2(\Omega))$
such that the equations in $\Sigma$ hold for a.a. $x\in \Omega$ and a.a. $t>0$, \eqref{neumann_control} and \eqref{neumann} hold for a.a. $t>0$, a.a. $x\in \partial \mathcal{W}$ and a.a. $x\in \partial \Omega \setminus \partial \mathcal{W}$, respectively, and $(T(x,0),S(x,0))=(T_o ,S_o )$ for a.a. $x\in \Omega$.
\end{definitionn}

\subsection{Control-design problem}
We introduce the deviation $\tilde{T} := T-T_a$. 
In this new variable, $\Sigma$ can be rewritten as
\footnote{By abuse of notation, we say that $(\tilde{T},S)$ is a solution to $\Sigma$ whenever $(\tilde{T}+T_a,S)$ is a solution to $\Sigma$, in the sense of Definition \ref{def1}.} 
\begin{equation*}
\Sigma \ : \ \left\lbrace
\begin{aligned}
\frac{\partial \tilde{T}}{\partial t} =&~ \varepsilon \Delta \tilde{T} - v\cdot \nabla \tilde{T}  + A\left(S\tilde{r}(\tilde{T})-C\tilde{T}\right) \\
\frac{\partial S}{\partial t} =&~ -C_S S\tilde{r}(\tilde{T}),
\end{aligned}
\right.
\end{equation*}
where 
\begin{equation*}
\tilde{r}(\tilde{T}) :=
\left\lbrace 
\begin{aligned}
&\exp^{-\gamma/\tilde{T}} \quad &\tilde{T}>0 \\
&0 \quad &\tilde{T}\leq 0.
\end{aligned}
\right.
\end{equation*}
Furthermore, \eqref{neumann_control} and \eqref{neumann} become, respectively, 
\begin{align*}
\frac{\partial \tilde{T}(x,t)}{\partial n } =&~ \kappa (\tilde{T}(x,t),x,t)\quad \text{for a.a. $x\in \partial \mathcal{W}$ and a.a. $t>0$},\\
\frac{\partial \tilde{T}(x,t)}{\partial \nu  } =&~ 0 \quad \text{for a.a. $x\in \partial \Omega \setminus \partial \mathcal{W}$ and a.a. $t>0$}.
\end{align*}

When the wind velocity $v$ is known, we will design $\kappa$ to guarantee the following property.
\begin{Property}\label{prop1}
Given $S_o\in L^2(\Omega;[0,1])$, there exists a constant $\alpha>0$ such that, if $(\tilde{T},S)$ is a solution to $\Sigma$ starting from $(\tilde{T}_o,S_o)$ with $\tilde{T}_o\in L^2(\Omega;\mathbb{R})$, we have 
\begin{align}
|\tilde{T}(\cdot,t)|_{L^2(\mathcal{W})}^2 \leq |\tilde{T}_o|_{L^2(\mathcal{W})}^2\exp^{-\alpha t} \quad \forall t\geq 0. \label{L2_decay}
\end{align}
\hfill $\bullet$
\end{Property}

When the wind velocity is unknown, we will redesign $\kappa$ adaptively to guarantee the following property.
\begin{Property}\label{prop2}
If $(\tilde{T},S)$ is a solution to $\Sigma$ starting from $(\tilde{T}_o,S_o) \in L^2(\Omega;\mathbb{R})\times L^2(\Omega;[0,1])$, then 
\begin{align}
|\tilde{T}|_{L^2(\mathcal{W})} \ \ \text{is bounded and} \ \ \lim_{t\to +\infty}|\tilde{T}(\cdot,t)|_{L^2(\mathcal{W})} = 0. \label{ineq_prop2}
\end{align}
\hfill $\bullet$
\end{Property}

\begin{Remark}
It is argued in \cite{model3} that the wind velocity $v$ can be estimated from atmospheric data or modeled using, e.g., the Navier-Stokes equations. However, both approaches introduce potential inaccuracies, highlighting the need for an adaptive redesign of $\kappa$.
\end{Remark}

\begin{Remark}
Properties \ref{prop1} and \ref{prop2} describe how the region $\mathcal{W}$ is protected from the wildfire, albeit in different ways depending on whether the wind velocity is known or not. In Property \ref{prop1}, we essentially guarantee that the region $\mathcal{W}$ is isolated from the wildfire, with the temperature within this region decaying exponentially towards the ambient temperature. On the other hand, in Property \ref{prop2}, when the wind velocity is unknown, we lose the guarantee of exponential decay of the temperature of $\mathcal{W}$ towards the ambient temperature. Without accurate knowledge of the wind velocity, we are not able, with our approach, to fully account for the heat advected from $\Omega \setminus \mathcal{W}$ towards $\mathcal{W}$. As a result, we no longer guarantee that the protected region $\mathcal{W}$ is isolated from the wildfire. However, we can still ensure that the temperature of $\mathcal{W}$ converges to the ambient temperature as time grows, mitigating the impact of the fire.
\end{Remark}

To guarantee Property \ref{prop1}, we require the following assumption.
\begin{assumption}\label{ass1}
We have
\begin{align}
2AC + \frac{2\varepsilon }{\sup_{x\in \mathcal{W}}|x|^2}-\frac{2A\exp^{-1}}{\gamma}\sup_{x\in \mathcal{W}}\left|S_o \right| > \mathcal{V}, \label{ineq_ass}
\end{align}
where $\mathcal{V} := \sup_{x\in \mathcal{W}}\left|\nabla \cdot v \right|$.
\end{assumption}

\begin{Remark}
Inequality \eqref{ineq_ass} is satisfied when the heat loss coefficient $C$ and the thermal diffusivity $\varepsilon$ are sufficiently large. Both coefficients directly influence the magnitude of the stabilizing terms in $\Sigma$, with $C$ enhancing heat dissipation, and $\varepsilon$ strengthening thermal diffusion. A more interesting aspect is the dependence of \eqref{ineq_ass} on $\mathcal{V}$, which can be interpreted as the largest value of the wind acceleration in $\mathcal{W}$. While Assumption \ref{ass1} does not impose any restrictions on the wind velocity, it does require that the wind acceleration remains relatively small compared to the other positive terms in \eqref{ineq_ass}. Intuitively, and informally speaking, a sudden gust or shift in wind direction seems to introduce more complexity to control the spread of wildfires compared to a steady wind.
\end{Remark}

To guarantee Property \ref{prop2}, we need, in addition to Assumption \ref{ass1}, the following condition. 
\begin{assumption}\label{ass2}
The wind velocity $v$ is bounded along $\partial \mathcal{W}$. That is, there exists a constant $\bar{v}\geq 0$ such that 
\begin{align}
|v(x,t)| \leq \bar{v} \quad \text{for all } (x,t)\in \partial \mathcal{W}\times (0,+\infty). \label{boundedness}  
\end{align}
\end{assumption}

We emphasize that the upperbound $\bar{v}$ on $v$ is not required to be known. 

\section{Main Results} \label{main results}
\subsection{Known Wind Velocity}\label{main1}
In this section, we suppose that the wind velocity $v$ along $\partial \mathcal{W}$ is known, and we show how to establish Property \ref{prop1}. 

We first introduce the feedback law $\kappa$, given by
\begin{align}
\kappa := \left[\left(\frac{n\cdot v-2k}{2\varepsilon}\right)-2\frac{\sup_{x\in \partial \mathcal{W}}|x|}{\sup_{x\in \mathcal{W}}|x|^2}\right] \tilde{T}, \label{feedback}
\end{align}
where $k\geq 0$ is a free control gain. We can note that $\kappa$ depends explicitly on $x$ and $t$ since it involves the wind velocity $v$. A physical interpretation of \eqref{feedback} is provided below.

\begin{Remark}
The feedback law $\kappa$ is constituted of two key terms. The first term, $\tilde{T}(n \cdot v)/(2\varepsilon)$, counteracts the heat advection across the boundary $\partial \mathcal{W}$ due to wind. Indeed, when $v\cdot n >0$, wind flows inward from $\Omega \setminus \mathcal{W}$ towards $\mathcal{W}$, bringing external heat into $\mathcal{W}$. In this case, our controller applies a negative heat flux (cooling) proportional to the incoming flow to counteract this external thermal influence. Conversely, when $v\cdot n <0$, wind flows outward, carrying heat away from $\mathcal{W}$, potentially making the temperature of $\mathcal{W}$ smaller than the ambient temperature. Here, the controller applies a positive heat flux (heating) to compensate for this heat loss, maintaining a thermal balance. The second term, $\tilde{T}(-k/\varepsilon - 2\sup_{x\in \partial \mathcal{W}}|x|/\sup_{x\in \mathcal{W}}|x|^2)$, provides damping regardless of wind direction, to regulate the temperature of $\mathcal{W}$ towards the ambient temperature.
\end{Remark}
\begin{Theorem}\label{thm1}
Suppose that Assumption \ref{ass1} holds, and let $\kappa$ be given by \eqref{feedback}. Then, Property \ref{prop1} is verified with 
\begin{align}
\hspace{-0.2cm}\alpha :=&~2AC+\frac{2\varepsilon}{\sup_{x\in \mathcal{W}}|x|^2}-\mathcal{V}-\frac{2A\exp^{-1}}{\gamma}\sup_{x\in \mathcal{W}}\left|S_o\right|. \label{alpha}
\end{align}
\end{Theorem}

\begin{Remark}
Our control strategy has two key features. First, it employs Neumann-type actuation, meaning that we regulate the heat flux at $\partial \mathcal{W}$ rather than directly imposing a temperature on $\partial \mathcal{W}$ (Dirichlet actuation). Neumann boundary control is usually more practical to implement than its Dirichlet counterpart when it comes to thermal systems \cite{heat}. E.g., firefighters positioned along $\partial \mathcal{W}$ can activate water sprays, while a control algorithm adjusts in real time their intensity to impose a desired heat flux. Second, our controller is \textit{decentralized} \cite{dec,dec2}, meaning that the value of $\kappa$ at each location $x\in \partial \mathcal{W}$ depends only on the values of $v$ and $T$ at that location. This is particularly advantageous in environments where state measurements are challenging or costly, such as in wildfire management \cite{detection}.
\end{Remark}
 
We follow a Lyapunov-based approach to prove Theorem$~$\ref{thm1}. Specifically, we consider the functional 
\begin{align}
B(\tilde{T}) := \frac{1}{2} |\tilde{T}|_{L^2(\mathcal{W})}^2. \label{barrier}
\end{align}
By differentiating $B$ with respect to time, we obtain 
\begin{align}
\dot{B}
=&~ \int_{\mathcal{W}}\tilde{T}\bigg[\varepsilon \Delta \tilde{T}-v\cdot \nabla \tilde{T}+A\big(S\tilde{r}(\tilde{T})-C\tilde{T}\big)\bigg]dx. \label{B_der_1}
\end{align}

In the next set of lemmas, we analyze each term at the right-hand side of \eqref{B_der_1} in order to derive an appropriate upperbound on $\dot{B}$. We start by evaluating the contribution of the diffusion term $\varepsilon \Delta \tilde{T}$. 
\begin{lemma}
Along the solutions to $\Sigma$, we have 
\begin{align}
\int_{\mathcal{W}}\tilde{T}\Delta \tilde{T}dx &\leq - \frac{2B}{\sup_{x\in \mathcal{W}}|x|^2} +\bigg(2\frac{\sup_{x\in \partial \mathcal{W}}|x|}{\sup_{x\in \mathcal{W}}|x|^2}\bigg)|\tilde{T}|_{L^2(\partial \mathcal{W})}^2\nonumber \\
&~+\int_{\partial \mathcal{W}} \frac{\partial \tilde{T}}{\partial n}\tilde{T}d\sigma, \label{diffusion2}
\end{align}
where $\sigma$ is the surface measure on $\partial \mathcal{W}$; see \cite[Appendix C]{evans} for the definition of $d\sigma$.
\end{lemma}
\begin{proof}
To establish \eqref{diffusion2}, we first use integration by parts to obtain 
\begin{align}
\int_{\mathcal{W}}\tilde{T}\Delta \tilde{T}&dx = \int_{\partial \mathcal{W}}\frac{\partial \tilde{T}}{\partial n}\tilde{T}d\sigma - |\nabla \tilde{T}|_{L^2(\mathcal{W})}^2. \label{first_integral}
\end{align}
According to Friedrichs-Poincaré's inequality \cite{fried}, we have 
\begin{align*}
|\tilde{T}|_{L^2(\mathcal{W})}^2 \leq&~ 2\sup_{x\in \partial \mathcal{W}}|x||\tilde{T}|_{L^2(\partial \mathcal{W})}^2 +\sup_{x\in \mathcal{W}}|x|^2|\nabla \tilde{T}|_{L^2(\mathcal{W})}^2,
\end{align*}
which implies that 
\begin{align}
-| \nabla \tilde{T}|_{L^2(\mathcal{W})}^2 
\leq&~ - \frac{|\tilde{T}|_{L^2(\mathcal{W})}^2}{\sup_{x\in \mathcal{W}}|x|^2} \nonumber \\
&~+\left(2\frac{\sup_{x\in \partial \mathcal{W}}|x|}{\sup_{x\in \mathcal{W}}|x|^2}\right)|\tilde{T}|_{L^2(\partial \mathcal{W})}^2.\label{g_fred}
\end{align}
We obtain \eqref{diffusion2} by combining \eqref{first_integral} and \eqref{g_fred}.
\end{proof}

Next, we evaluate the contribution of advection. 
\begin{lemma}
Along the solutions to $\Sigma$, we have 
\begin{align}
-\int_{\mathcal{W}}\tilde{T} v\cdot \nabla \tilde{T}dx \leq -\frac{1}{2}\int_{\partial \mathcal{W}}\tilde{T}^2 v\cdot nd\sigma + \mathcal{V}B.  \label{advection}
\end{align}
\end{lemma}
\begin{proof}
To establish \eqref{advection}, we first rewrite its left-hand side as 
\begin{align*}
\int_{\mathcal{W}}\tilde{T} v\cdot \nabla \tilde{T}dx =&~ \frac{1}{2}\int_{\mathcal{W}}v_1\frac{\partial (\tilde{T}^2)}{\partial x_1}dx+\frac{1}{2}\int_{\mathcal{W}}v_2\frac{\partial (\tilde{T}^2)}{\partial x_2}dx.
\end{align*}
Using integration by parts, we obtain 
\begin{align*}
\int_{\mathcal{W}}v_1 \frac{\partial \tilde{T} ^2}{\partial x_1}dx =&~ \int_{\partial \mathcal{W}}\tilde{T} ^2v_1 n_1 d\sigma  - \int_{\mathcal{W}}\tilde{T} ^2\frac{\partial v_1 }{\partial x_1}dx.
\end{align*}
Similarly, we have 
\begin{align*}
\int_{\mathcal{W}}v_2 \frac{\partial \tilde{T} ^2}{\partial x_2}dx =&~ \int_{\partial \mathcal{W}}\tilde{T} ^2v_2 n_2 d\sigma - \int_{\mathcal{W}}\tilde{T} ^2\frac{\partial v_2 }{\partial x_2}dx.
\end{align*}
As a consequence, 
\begin{align*}
\int_{\mathcal{W}}\tilde{T}  v \cdot \nabla \tilde{T} dx
=&~\frac{1}{2}\int_{\partial \mathcal{W}}\tilde{T} ^2 v \cdot n d\sigma-\frac{1}{2}\int_{\mathcal{W}}\tilde{T} ^2 \left(\nabla \cdot v \right)dx,
\end{align*}
which implies \eqref{advection}.
\end{proof}

Finally, we analyze the contribution of the term $AS\tilde{r}(\tilde{T})$ coming from the Arrhenius law. 
\begin{lemma}
Along the solutions to $\Sigma$, we have 
\begin{align}
\int_{\mathcal{W}}S \tilde{T} &\tilde{r}(\tilde{T} )dx \leq \frac{2\exp^{-1}}{\gamma}\sup_{x\in \mathcal{W}}\left|S_o \right|B. \label{reaction}
\end{align}
\end{lemma}
\begin{proof}
Suppose that $\tilde{T}>0$. Then, $\tilde{T}\tilde{r}(\tilde{T}) = \tilde{T}\exp^{-\gamma/\tilde{T}}$. Hence, to establish \eqref{reaction}, we show that 
\begin{align}
\tilde{T}\exp^{-\gamma/\tilde{T}} \leq \frac{\exp^{-1}}{\gamma}\tilde{T}^2. \label{to_show}
\end{align}
To do so, we study the function
$f(s) := \frac{\exp^{-\gamma/s}}{s}$ for $s>0$. By differentiating $f$, we obtain $f'(s) = \frac{\gamma/s-1}{s^2}\exp^{-\gamma/s}$. We note, in particular, that $f'(s)=0$ if and only if $s=\gamma$. To verify if the latter is a maximum or a minimum point, we differentiate $f'$ to get $f''(s) = \left(\frac{-3\gamma +2s + \gamma^2/s-\gamma}{s^4}\right)\exp^{-\gamma/s}$. By evaluating $f''$ at $s=\gamma$, we obtain $f''(\gamma) = -(1/\gamma^3)\exp^{-1} < 0$.
As a result, $s=\gamma$ is a maximum point for $f$, and $f(s) \leq f(\gamma) = (1/\gamma)\exp^{-1}$ for all $s>0$.
In particular, 
\begin{align*}
\frac{\tilde{T}\exp^{-\gamma/\tilde{T}}}{\tilde{T}^2} = \frac{\exp^{-\gamma/\tilde{T}}}{\tilde{T}} \leq \frac{\exp^{-1}}{\gamma},
\end{align*}
which further implies \eqref{to_show}. Hence, using the fact that $S(x,t) \leq S_o$ for all $x\in \mathcal{W}$ and all $t\geq 0$, we get \eqref{reaction}. 
\end{proof}

Now, combining \eqref{B_der_1}, \eqref{diffusion2}, \eqref{advection}, and \eqref{reaction}, we obtain
\begin{align}
\dot{B} \leq&~ -\alpha B +\varepsilon \int_{\partial \mathcal{W}} \kappa \tilde{T} d\sigma -\int_{\partial \mathcal{W}}\left(\frac{n  \cdot v }{2}\right)\tilde{T} ^2d\sigma  \nonumber \\
&~+ 2\varepsilon \frac{\sup_{x\in \partial \mathcal{W}}|x|}{\sup_{x\in \mathcal{W}}|x|^2}|\tilde{T}|_{L^2(\partial \mathcal{W})}^2. \label{B_der_2}
\end{align}
As a result, under our choice of feedback law in \eqref{feedback}, we get
\begin{align}
\dot{B} \leq - \alpha B - k|\tilde{T}|_{L^2(\partial \mathcal{W})}^2 ,
\end{align}
which implies \eqref{L2_decay}.

\subsection{Unknown Wind Velocity}\label{main2}
We suppose in this section that the wind velocity $v$ is unknown, and we show how to guarantee Property \ref{prop2}. We propose redesign $\kappa$ as
\begin{align}
\kappa := -\left[\left(\frac{\hat{v}+2k}{2\varepsilon}\right)+2\frac{\sup_{x\in \partial \mathcal{W}}|x|}{\sup_{x\in \mathcal{W}}|x|^2}\right] \tilde{T}, \label{new_kappa}
\end{align}
where $\hat{v} : \partial \mathcal{W}\times (0,+\infty) \to \mathbb{R}$ is governed by
\begin{align}
\frac{\partial \hat{v}(x,t)}{\partial t} &= \frac{\lambda}{2}\tilde{T}(x,t)^2 \quad x\in \partial \mathcal{W}, \ \ t>0, \label{adaptation}\\
\hat{v}(x,0) &= \hat{v}_o(x)  \quad x\in \partial \mathcal{W}, \label{init}
\end{align}
with $\lambda >0$ an adaptation gain and $\hat{v}_o:\partial \mathcal{W}\to \mathbb{R}_{+}$.

\begin{Theorem}\label{thm2}
Suppose that Assumptions \ref{ass1} and \ref{ass2} hold, and let $\kappa$ be given by \eqref{new_kappa}, with $\hat{v}$ governed by \eqref{adaptation}-\eqref{init}. Then, Property \ref{prop2} is verified.
\end{Theorem}
\begin{proof}
To prove Theorem \ref{thm2}, we consider the functional 
\begin{align}
\mathcal{Z}(\tilde{T},\hat{v}) := B(\tilde{T}) + \frac{1}{2\lambda} |\bar{v}-\hat{v}|_{L^2(\partial \mathcal{W})}^2. \label{invariance1}
\end{align}
By differentiating $\mathcal{Z}$ with respect to time, we obtain 
\begin{align*}
\dot{\mathcal{Z}} = \dot{B} - \frac{1}{\lambda}\int_{\partial \mathcal{W}}(\bar{v}-\hat{v} )\frac{\partial \hat{v} }{\partial t} d\sigma.
\end{align*}
Furthermore, using \eqref{B_der_2}, note that we have 
\begin{align*}
\dot{\mathcal{Z}} \leq&~ -\alpha B +\varepsilon \int_{\partial \mathcal{W}} \kappa \tilde{T} d\sigma -\int_{\partial \mathcal{W}}\left(\frac{n  \cdot v }{2}\right)\tilde{T} ^2d\sigma \nonumber \\
&~ + 2\varepsilon \frac{\sup_{x\in \partial \mathcal{W}}|x|}{\sup_{x\in \mathcal{W}}|x|^2}|\tilde{T}|_{L^2(\partial \mathcal{W})}^2 - \frac{1}{\lambda}\int_{\partial \mathcal{W}}(\bar{v}-\hat{v} )\frac{\partial \hat{v} }{\partial t} d\sigma. 
\end{align*}
Next, since $|n|=1$, then $-\int_{\partial \mathcal{W}}\left(\frac{n  \cdot v }{2}\right)\tilde{T} ^2d\sigma \leq \frac{\bar{v}}{2}|\tilde{T}|_{L^2(\partial \mathcal{W})}^2$. As a consequence, we obtain 
\begin{align*}
\dot{\mathcal{Z}} \leq&~ - \alpha B + \frac{1}{\lambda} \int_{\partial \mathcal{W}}\hat{v} \frac{\partial \hat{v} }{\partial t}d\sigma +\bar{v} \int_{\partial \mathcal{W}}\left(\frac{\tilde{T} ^2}{2}- \frac{1}{\lambda}\frac{\partial \hat{v} }{\partial t}\right)d\sigma \\
&~+\varepsilon \int_{\partial \mathcal{W}} \kappa \tilde{T} d\sigma +2\varepsilon \frac{\sup_{x\in \partial \mathcal{W}}|x|}{\sup_{x\in \mathcal{W}}|x|^2}|\tilde{T}|_{L^2(\partial \mathcal{W})}^2.
\end{align*}
Under the adaptation law in \eqref{adaptation}, we get 
\begin{align*}
\dot{\mathcal{Z}} \leq&~ -\alpha B + \frac{1}{2}\int_{\partial \mathcal{W}}\hat{v} \tilde{T} ^2d\sigma + \varepsilon \int_{\partial \mathcal{W}} \kappa \tilde{T} d\sigma \\
&~+2\varepsilon \frac{\sup_{x\in \partial \mathcal{W}}|x|}{\sup_{x\in \mathcal{W}}|x|^2}|\tilde{T}|_{L^2(\partial \mathcal{W})}^2.
\end{align*}
As a result, under our choice of $\kappa$ in \eqref{new_kappa}, we obtain 
\begin{align}
\dot{\mathcal{Z}} \leq - \alpha B-k|\tilde{T}|_{L^2(\partial \mathcal{W})}^2. \label{invariance2}
\end{align}
The latter implies that $\mathcal{Z}$ is non-increasing. In particular, since $\mathcal{Z}\geq 0$ from \eqref{invariance1}, then, $\lim_{t\to +\infty} \mathcal{Z}(t)$ exists and is finite. It further implies that
\begin{align}
\int_{0}^{+\infty}\dot{\mathcal{Z}}(t)dt < +\infty.\label{invariance3}
\end{align}
Combining \eqref{invariance2} and \eqref{invariance3}, we conclude that $\int_{0}^{+\infty}B(t)dt \leq -\frac{1}{\alpha}\int_{0}^{+\infty}\dot{\mathcal{Z}}(t)dt < +\infty$.
Since $B\geq 0$ and $\int_{0}^{+\infty}B(t)dt< +\infty$, then, by the alternative to Barbalat's lemma in \cite[Lemma 3.1]{dec}, it is enough to show that $\dot{B}\leq M$ for some $M\in \mathbb{R}$ to conclude that $B$ converges asymptotically to zero. To establish the boundedness of $\dot{B}$ from above, first note that, according to \eqref{B_der_2} and \eqref{new_kappa}, we have $\dot{B} \leq \int_{\partial \mathcal{W}}\left(\frac{\bar{v}-\hat{v} }{2}\right)\tilde{T} ^2d\sigma$. Since $\hat{v}$ is non-decreasing from \eqref{adaptation}, and since we have chosen $\hat{v}_o \geq 0$ for all $x\in \partial \mathcal{W}$ in \eqref{init}, then $\hat{v}(x,t)\geq 0$ for all $x\in \partial \mathcal{W}$ and all $t\geq 0$. As a consequence, $\dot{B} \leq \frac{\bar{v}}{2}|\tilde{T}|_{L^2(\partial \mathcal{W})}^2 \leq \bar{v}B$. Now, using \eqref{invariance1} and the fact that $\mathcal{Z}$ is non-increasing, we conclude that $B$ is bounded, which further implies that $\dot{B}$ is bounded from above. As a result, $B$ converges asymptotically to zero.
\end{proof}

\section{Simulation Results}\label{simu}
In this section, we illustrate our results via numerical simulations performed in MATLAB. The numerical method is described in the Appendix of \cite{long}. It relies on a central difference scheme for space discretization of $\Delta T$, an Euler backward scheme for $\nabla T$ at the interior of $\Omega/\partial \mathcal{W}$, and an Euler forward scheme for time discretization. The boundary conditions are discretized using Euler forward and backward schemes. The time step is $\Delta t = 0.01$.

We let $\Omega = [0,L_1]\times [0,L_2]$ and $\mathcal{W} = [2L_1/3,L_1]\times [0,L_2]$, where $L_1,L_2>0$. We use an $80\times 80$ regular mesh for $\Omega$. Furthermore, we let $v= \begin{bmatrix} 1 & 0 \end{bmatrix}$.
A schematic representation of $\Omega$, $\mathcal{W}$, $\nu$, $n$, and $v$, is shown below.

\begin{tikzpicture}[
    >={Stealth[length=2mm]},
    vector/.style={->, thick, blue},
    normal/.style={->, thick, red},
    boundary/.style={thick},
    wind/.style={->, thick, green!50!black}
]

\def\L{6} 
\def\H{4} 
\def\WL{2*\L/3} 

\draw[boundary] (0,0) rectangle (\L,\H);
\node at (\L/4,\H/4) {$\Omega \setminus \mathcal{W}$};

\draw[boundary, fill=gray!20] (\WL,0) rectangle (\L,\H);
\node at (5*\L/6,\H/4) {$\mathcal{W}$};

\draw[normal] (0,\H/2) -- +(-0.5,0) node[left] {$\nu$};
\draw[normal] (\L/3,0) -- +(0,-0.5) node[below] {$\nu$};
\draw[normal] (\L/3,\H) -- +(0,0.5) node[above] {$\nu$};

\draw[normal] (\WL,\H/2) -- +(-0.5,0) node[left] {$n$};
\draw[normal] (5*\L/6,0) -- +(0,-0.5) node[below] {$n$};
\draw[normal] (5*\L/6,\H) -- +(0,0.5) node[above] {$n$};
\draw[normal] (\L,\H/2) -- +(0.5,0) node[right] {$n$};

\draw[wind] (\L/4,\H/2) -- +(0.7,0) node[right,font=\small] {$v$};

\node[below left] at (0,0) {$(0,0)$};
\node[below right] at (5.7,0) {$(L_1,0)$};
\node[above left] at (0,\H) {$(0,L_2)$};
\node[above right] at (5.7,\H) {$(L_1,L_2)$};
\node[below] at (\WL,0) {$(\frac{2L_1}{3},0)$};

\end{tikzpicture}

The inequality in Assumption \ref{ass1} becomes 
\begin{align}
2AC + \frac{2\varepsilon }{L_1^2+L_2^2}-\frac{2A\exp^{-1}}{\gamma}\sup_{x\in \mathcal{W}}\left|S_o \right| > 0. \label{ass_e}
\end{align}
We select $A=1.8793\times 10^2,\ \ C=7.2558\times 10^{-4}, \ \
\varepsilon =2.1360\times 10^{-1}, \ \ \gamma = 5.5849\times 10^2$. Additionally, $S_o =1$, and $L_1=L_2=50(m)$. We can note that Assumption \ref{ass1} is verified since the left-hand side of \eqref{ass_e} is $2.52\times 10^{-2}$. 

As in \cite{model3}, we let the ambient temperature be given by $T_a = 300(K)$. According to Theorem \ref{thm1}, since Assumption \ref{ass1} holds, then Property \ref{prop1} is verified, with $\alpha := 2.52\times 10^{-2}$, if we set $\kappa$ as in \eqref{feedback}. We select the control gain $k:=1$ for simulations. Moreover, the initial temperature distribution is given by $T(x,0) = T_o  = T_a + T_c\exp^{\frac{-|x-\bar{x}|^2}{2w^2}}$ for all $x\in \Omega$, where $T_c =1000(K)$, $w:=10$, and $\bar{x} = \begin{bmatrix} 25 & 25\end{bmatrix}(m)$.

\begin{figure}
    \centering
    \includegraphics[width=\linewidth]{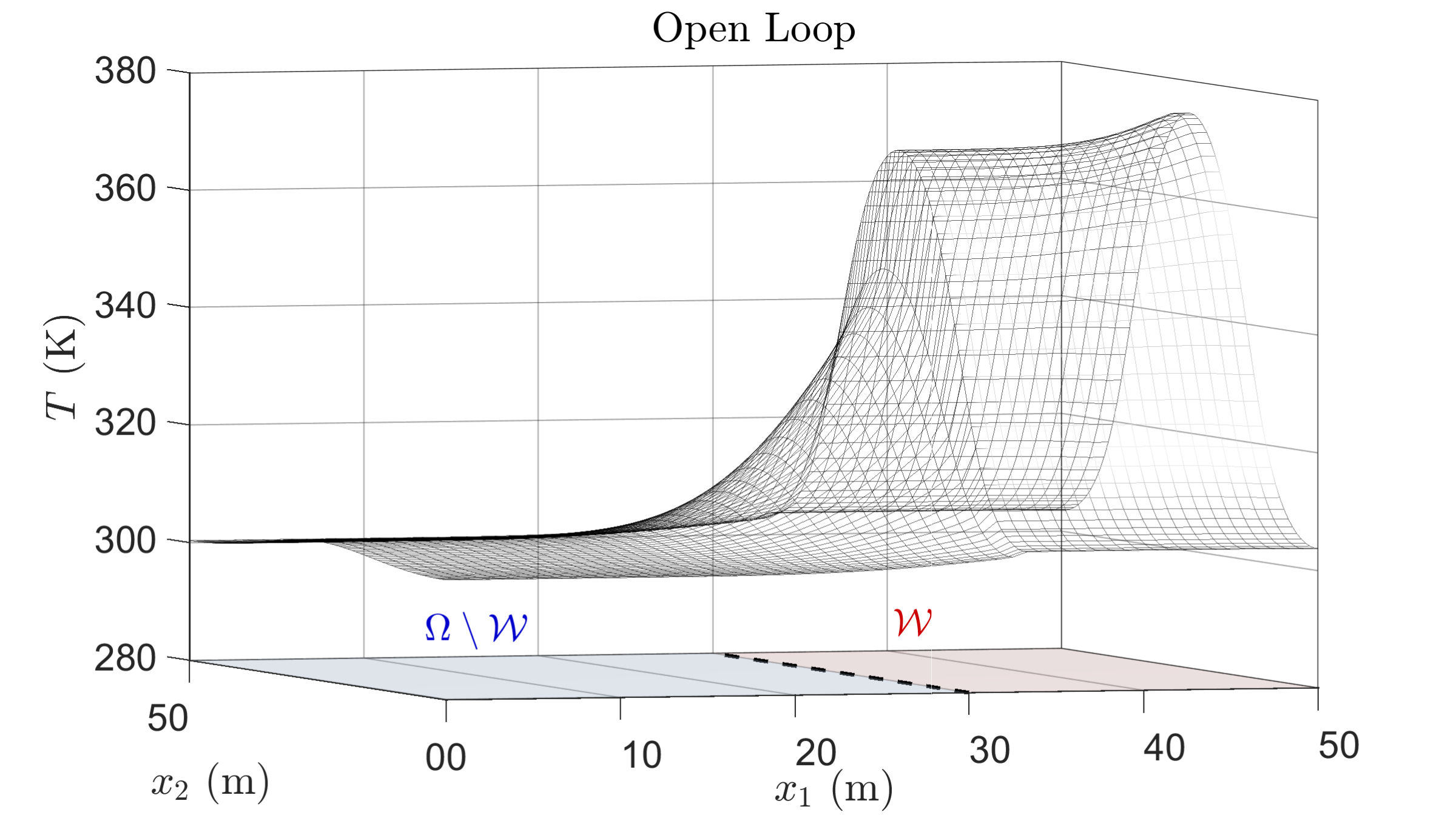} 
    \includegraphics[width=\linewidth]{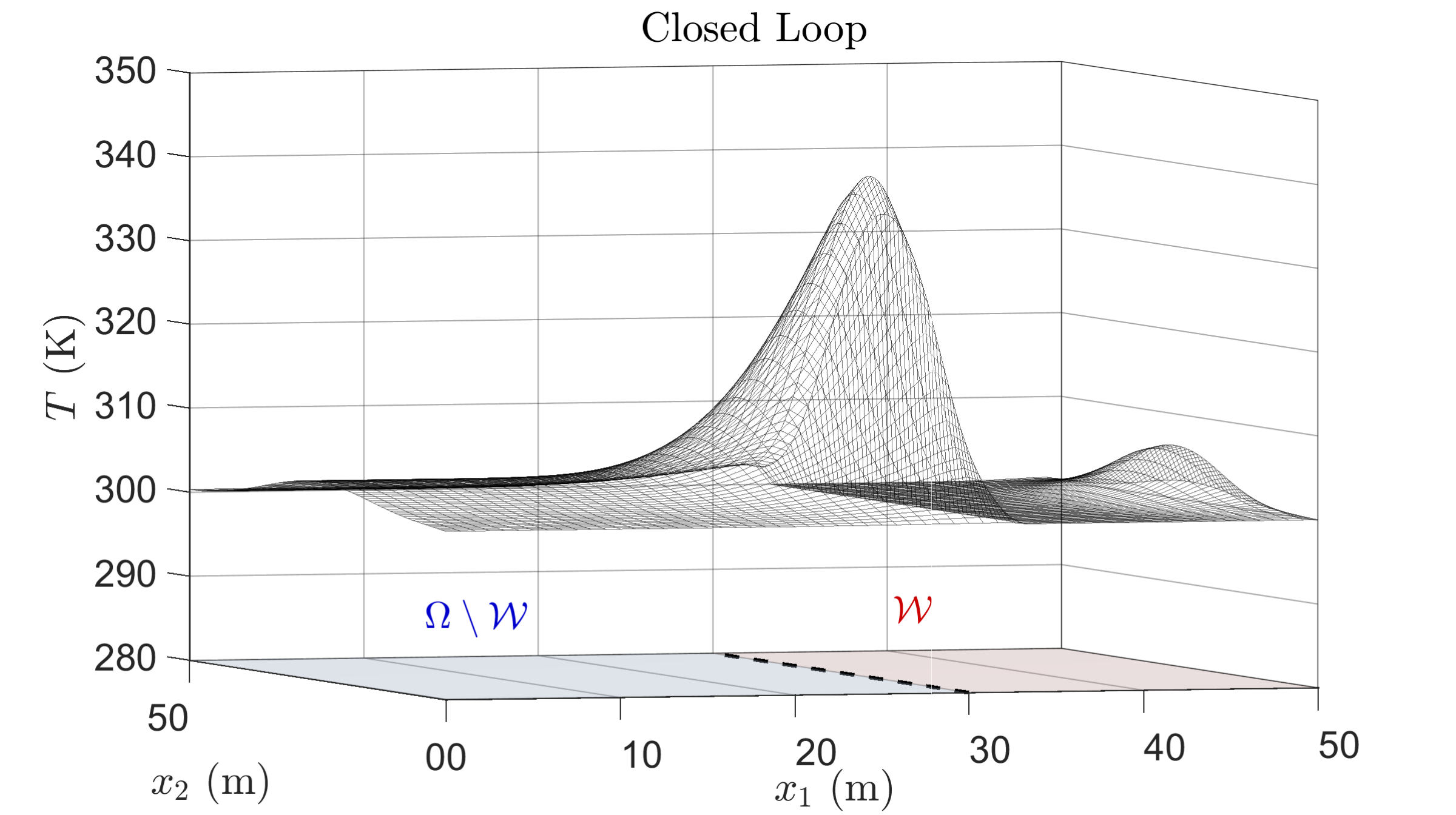}
    \caption{Response of $\Sigma$ to $\kappa:=0$ (top) vs response of $\Sigma$ to \eqref{feedback} with $k=1$ (bottom), at $t=20s$.}
    \label{fig1}
\end{figure}

\begin{figure}
    \centering
    \includegraphics[width=\linewidth]{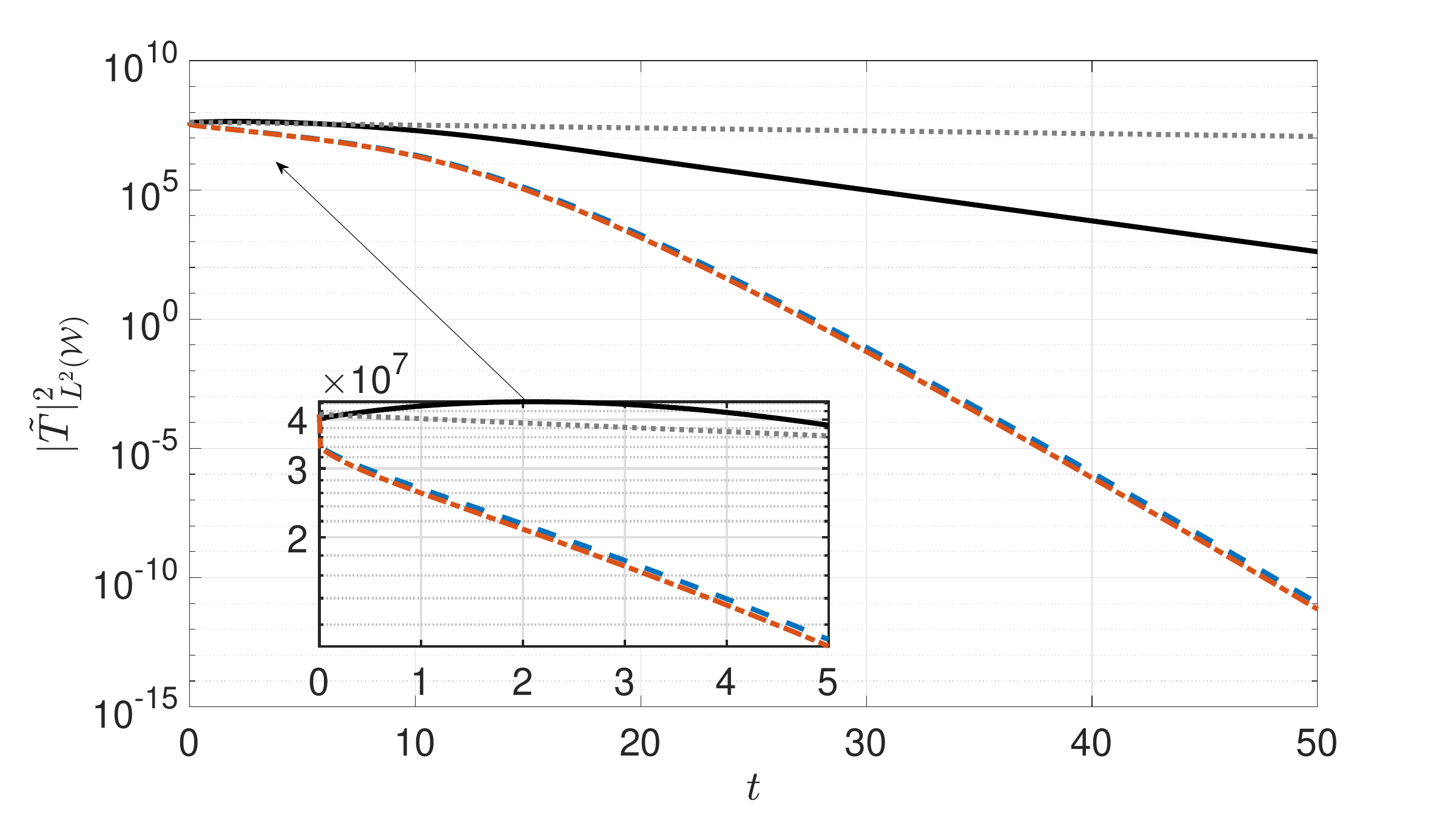}
    \caption{$|\tilde{T}|_{L^2(\mathcal{W})}^2$ under $\kappa:=0$ (black) vs under \eqref{feedback} (blue) vs under \eqref{new_kappa}-\eqref{adaptation}-\eqref{init} (red) vs the theoretical bound at the right-hand side of \eqref{L2_decay} (gray).}
    \label{barrier_2}
\end{figure}
At the top of Figure \ref{fig1}, we simulate the response of $\Sigma$, at $t=20s$, to $\kappa := 0$. We can note that the temperature within the region $\mathcal{W}$ is above $T_a = 300(K)$. When closing the loop with \eqref{feedback}, the response of $\Sigma$ at $t=20(s)$ is shown at the bottom of Figure \ref{fig1}. We can note that, because of our controller, the temperature within $\mathcal{W}$ converges towards the ambient temperature $T_a=300(K)$. This is expected from Theorem \ref{thm1}, since $\tilde{T}=T-T_a$ converges to zero in the $L^2$ norm according to \eqref{L2_decay}. To observe this convergence, we plot in Figure \ref{barrier_2} the map $t\mapsto |\tilde{T}(\cdot,t)|_{L^2(\mathcal{W})}^2$ under $\kappa:=0$, then under \eqref{feedback}. We also plot the theoretical bound on $|\tilde{T}|_{L^2(\mathcal{W})}^2$ at the right-hand side of \eqref{L2_decay}. As expected, \eqref{L2_decay} is verified. Moreover, under $\kappa:=0$, although $|\tilde{T}|_{L^2(\mathcal{W})}^2$ decays towards zero, it first grows exceeding the right-hand side of \eqref{L2_decay} before starting to decay. Moreover, its decay is slower than when \eqref{feedback} is applied. Next, we apply the adaptive controller in \eqref{new_kappa}, with $\hat{v}$ governed by \eqref{adaptation}-\eqref{init}. We select the control gain $k=1$, the adaptation gain $\lambda = 0.1$, and the initial condition $\hat{v}_o(x)  = 0$ for all $x\in \partial \mathcal{W}$. According to Theorem \ref{thm2}, the $L^2$ norm of $\tilde{T}$ over $\mathcal{W}$ should converge asymptotically to zero, which is confirmed in Figure \ref{barrier_2}. 

Finally, we consider the case where heat dissipation to the surroundings is negligible, i.e., $C = 0$. Under this condition, Assumption \ref{ass1} is not satisfied. However, Figure \ref{energy_last} still shows that the temperature of $\mathcal{W}$ converges to $T_a = 300(K)$ under both \eqref{feedback} and \eqref{new_kappa}-\eqref{adaptation}-\eqref{init}, whereas under $\kappa:=0$, the $L^2$ norm of $\tilde{T}$ within $\mathcal{W}$ is increasing. These simulation results suggest that Assumption \ref{ass1}, while being sufficient to guarantee the asymptotic convergence of $t\mapsto |\tilde{T}(\cdot,t)|_{L^2(\mathcal{W})}^2$ to zero under our controllers, may not be necessary.

\begin{figure}
    \centering
    \includegraphics[width=\linewidth]{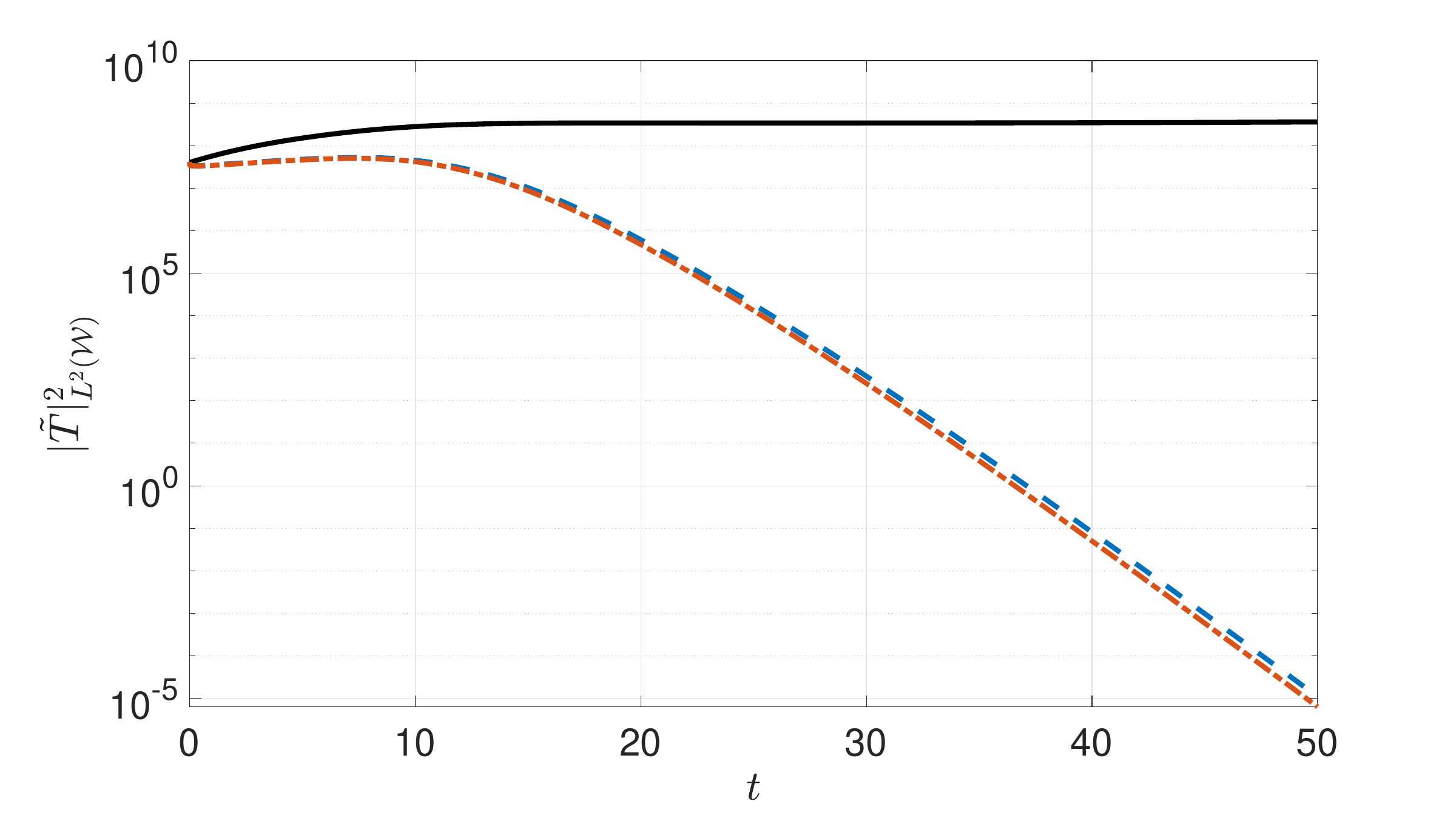}
    \caption{$|\tilde{T}|_{L^2(\mathcal{W})}^2$ under $\kappa:=0$ (black) vs under \eqref{feedback} (blue) vs under \eqref{new_kappa}-\eqref{adaptation}-\eqref{init} (red), when $C=0$.}
    \label{energy_last}
\end{figure}

\section{Conclusion and Research Perspectives}
In this paper, we considered a well-established model for heat propagation and fuel depletion in wildfires, for which we introduced a feedback-design strategy for regulating the temperature of a given subregion towards the ambient temperature. Specifically, for known wind velocity, we designed a Neumann-type controller guaranteeing the exponential convergence of the temperature to the ambient temperature. For unknown wind velocity, we developed an adaptive controller ensuring the asymptotic convergence of the temperature to the ambient level. An important direction would be to achieve convergence of $|\tilde{T}|_{L^2(\mathcal{W})}$ to zero at a prescribed rate. It would also be useful to enforce finite-time convergence of $S$ to zero within $\mathcal{W}$, at a prescribed settling time. Moreover, a relevant direction is to evaluate the amount of water required to achieve the control objectives under realistic wildfire profiles and to study scenarios where the irrigation rate is limited. Finally, applying the feedback law through piecewise-constant inputs and incorporating sampled-data or event-triggered control, as well as accounting for actuator intermittency, represent important directions for future research.

\section{Acknowledgment}
This work has been supported by the LabEx EnergyAlps2025.

\end{document}